\documentclass[11pt]{amsart}

\usepackage{amsmath,amsthm,amsfonts,amssymb,latexsym,mathrsfs,color,cite,tikz,cases,url,enumerate}
\usepackage{color}
\usepackage[pagebackref]{hyperref}
\usepackage[capitalize]{cleveref}
\usepackage{enumitem}
\usepackage{verbatim}
\usepackage{doi}

\def\lk{\mathrm{lk}}
\def\RR{\mathbb{R}}
\def\KK{\mathbb{K}}
\allowdisplaybreaks

  \makeatletter
  \newcommand{\numberlike}[2]{%
     \expandafter\def\csname c@#1\endcsname{%
         \expandafter\csname c@#2\endcsname}%
  }
  \makeatother

  \def\DefaultNumberTheoremWithin{section}

  \theoremstyle{plain}
  \newtheorem{lemma}{Lemma}
     \numberwithin{lemma}{\DefaultNumberTheoremWithin}
     \labelformat{lemma}{Lemma~#1}
  \newtheorem{theorem}{Theorem}
     \numberwithin{theorem}{\DefaultNumberTheoremWithin}
     \numberlike{theorem}{lemma}
     \labelformat{theorem}{Theorem~#1}
  \newtheorem*{maintheorem*}{Corollary of Main Theorem}
  \newtheorem{corollary}{Corollary}
     \numberwithin{corollary}{\DefaultNumberTheoremWithin}
     \numberlike{corollary}{lemma}
     \labelformat{corollary}{Corollary~#1}
  \newtheorem{proposition}{Proposition}
     \numberwithin{proposition}{\DefaultNumberTheoremWithin}
     \numberlike{proposition}{lemma}
     \labelformat{proposition}{Proposition~#1}
  \newtheorem{conjecture}{Conjecture}
     \numberwithin{conjecture}{\DefaultNumberTheoremWithin}
     \numberlike{conjecture}{lemma}
     \labelformat{conjecture}{Conjecture~#1}

  \theoremstyle{definition}
  
     \numberwithin{definition}{\DefaultNumberTheoremWithin}
     \numberlike{definition}{lemma}
     \labelformat{definition}{Definition~#1}
  \newtheorem{question}{Question}
     \numberwithin{question}{\DefaultNumberTheoremWithin}
     \numberlike{question}{lemma}
     \labelformat{question}{Question~#1}
  
     \numberwithin{problem}{\DefaultNumberTheoremWithin}
     \numberlike{problem}{lemma}
     \labelformat{problem}{Problem~#1}
   
  \theoremstyle{remark}
  \newtheorem{remark}{Remark}
     \numberwithin{remark}{\DefaultNumberTheoremWithin}
     \numberlike{remark}{lemma}
     \labelformat{remark}{Remark~#1}
  
     \numberwithin{example}{\DefaultNumberTheoremWithin}
     \numberlike{example}{lemma}
     \labelformat{example}{Example~#1}

     \numberwithin{claim}{\DefaultNumberTheoremWithin}
     \numberlike{claim}{lemma}
     \labelformat{claim}{Claim~#1}
\labelformat{figure}{Figure~#1}
\labelformat{chapter}{Chapter~#1}
\labelformat{appendix}{Appendix~#1}
\labelformat{section}{Section~#1}
\labelformat{subsection}{Subsection~#1}

\title[Real-rooted polynomials and $f$-vectors of simplicial complexes]{On a question about real-rooted polynomials and $f$-polynomials of simplicial complexes}

\author{Lili Mu} 
\address{School of Mathematics and Statistics, Jiangsu Normal University, Xuzhou 221116, PR China} 
\email{lilimu@jsnu.edu.cn}

\author{Volkmar Welker}
\address{Philipps-Universit\"at Marburg, Fachbereich Mathematik und Informatik, 35032 Marburg, Germany}
\email{welker@mathematik.uni-marburg.de}

\begin{document}

\begin{abstract}
    For a polynomial $f(t) = 1+f_0t+\cdots +f_{d-1}t^d$ with positive integer
    coefficients Bell and Skandera [Discrete Math, 307 (2007), pp. 668-682] ask if real-rootedness of
    $f(t)$ implies that there is a simplicial complex with
    $f$-vector $(1,f_0,\ldots,f_{d-1})$ or equivalently a simplicial complex with $f$-polynomial $f(t)$. In this paper we
    discover properties implied by the 
    real-rootedness of $f(t)$ in terms 
    of the
    binomial representation $f_i = \binom{x_{i+1}}{i+1}$, $i \geq 0$. We use these to partially answer the question by Bell and Skandera. We also describe two further approaches to the question and use one to verify that some well studied real-rooted classical polynomials are $f$-polynomials.

    Finally, we provide a series of results showing that the set of $f$-vectors of simplicial complexes is closed under constructions also preserving real-rootedness
    of their generating polynomials.
    \end{abstract}
\maketitle
\vspace{-0.5in}

\section{Introduction}

The results of this paper are motivated by a question of Bell and Skandera.

\begin{question}[Bell and Skandera, Question 1.1 in \cite{BS07}] \label{qu:bs} Let $f(t) = 1 + f_0t + \cdots + f_{d-1}t^d$ be a polynomial with positive integer coefficients. If $f(t)$ is real-rooted, is $(1,f_0,\ldots, f_{d-1})$ the
$f$-vector $(f_{-1}^ \Delta,\ldots, f_{d-1}^\Delta)$ of a simplicial complex $\Delta$?
\end{question}

Recall that for a simplicial complex $\Delta$ the entries $f_i^\Delta$ of the $f$-vector are the number of $i$-dimensional faces of $\Delta$.
Thus in other words the conclusion of \ref{qu:bs} is that 
$f(t) = f^ \Delta(t)$ is the $f$-polynomial of a simplicial complex $\Delta$. 
We develop three approaches to this question,
which may be of independent
interest, and collect further
evidence for a positive answer to the question. 
The main results of the paper are.
\begin{itemize}
    \item For a real-rooted polynomial $f(t) = 1 + f_0t + \cdots + f_{d-1}t^d$ with positive integer coefficients and real
    numbers $x_i \geq i$ such that
    $f(t)=1+\binom{x_1}{1}t+\cdots+\binom{x_d}{d}t^{d}$ we have
    $x_1 \geq \cdots \geq x_d$ (\ref{xk-decreasing}). As a consequence 
    we provide in \ref{cor:hvector} a new proof of \cite[Theorem 4.2]{BS07} showing that
    real-rootedness of $f(t)$ implies that the coefficient sequence is the $f$-vector
    of a multicomplex or equivalently the $h$-vector of a Cohen-Macaulay simplicial complex.
    \item If $f(t)=1+\binom{x_1}{1}t+\cdots+\binom{x_d}{d}t^{d}$ is a 
    real-rooted polynomial
    with positive integer coefficients, and
    $x_i \geq \lceil x_{i+1} \rceil$ then there is a simplicial complex
    $\Delta$ such that $f(t)$ is the $f$-polynomial of $\Delta$ (\ref{thm:main}).
    \item A method based on Murai's
    \cite{Mur10} admissible vectors which can be
    used to demonstrate that classical real-rooted
    polynomials (Eulerian, Stirling, derangement)
    are $f$-polynomials of simplicial complexes.
    For Eulerian polynomials this is a known result (see \cite{Gash98, ER94}).
    \item Proofs that classical constructions preserving real-rootedness 
    applied to $f$-poly\-no\-mi\-als $f^\Delta(t)$ of simplicial complexes yield the $f$-polynomial of another 
    simplicial complex (see \ref{sec:construction}).
\end{itemize}

\section{Notation, definitions and basic results}
For a finite set $\Omega$, {\it a simplicial complex} over ground set $\Omega$ is
a set of subsets of $\Omega$ such that $\sigma \subseteq \tau \in \Delta$
implies $\sigma \in \Delta$.
We call a $\sigma \in \Delta$ {\it a face} of $\Delta$
and {\it the dimension of a face} $\sigma$ is $\dim(\sigma) = \# \sigma -1$.
{\it The dimension of $\Delta$} is $\max_{\sigma \in \Delta} \dim(\sigma)$.
If $\Delta$ is a $d$-dimensional simplicial complex then the
vector $f^\Delta = (f_{-1}^ \Delta,\ldots, f_{d}^ \Delta)$ where
$f_i^ \Delta = \# \big\{\, \sigma \in \Delta~\big|~\dim(\sigma) = i\,\big\}$
is called the {\it $f$-vector} of $\Delta$
and $f^ \Delta(t)= f_{-1}^ {\Delta}+ f_{0}^ {\Delta} t+\cdots+ f_{d}^ {\Delta} t^ {d+1}$ is called the {\it $f$-polynomial} of $\Delta$. 
Expanding $(1-t)^d \cdot f^ \Delta\big(\,\frac{1}{1-t}\,\big) = h^ \Delta(t)=h_0^ \Delta +h_1^ \Delta t+\cdots+h_{d+1}^ \Delta t^ {d+1}$ yields the {\it $h$-polynomial}
with coefficient sequence $h^\Delta = (h_0^ \Delta,\ldots, h_{d+1}^ \Delta)$ {\it the $h$-vector} of $\Delta$.
For a face $F\in \Delta$, denote by $\lk_{\Delta}(F)=\big\{\,G\,:\, F\cup G\in \Delta, F \cap G = \emptyset\,\big\}$ {\it the link} of $F$ in $\Delta$.

It is well-known that the $f$-vector of a simplicial complex can be characterized in terms of functions based upon the following expression of a positive integer
$m$ as a sum of binomial coefficients.
Given a positive integer $i$, there exists a unique representation of $m$ in the form
$$m=\binom{a_i}{i}+\binom{a_{i-1}}{i-1}+\cdots+\binom{a_j}{j},$$
where $a_i>a_{i-1}>\cdots>a_j\geq j\geq 1$.
We refer to this representation of $m$ as the {\it $i$\textsuperscript{th} binomial expansion} of $m$.
Since the binomial representation of $m$ is unique it follows that

\begin{align}  \label{eq:binomial}
    m = \binom{a_i}{i}+\binom{a_{i-1}}{i-1}+\cdots+\binom{a_j}{j} & <  \binom{a_i+1}{i}.
    \end{align}

Define the functions {\it $(\mu_i)_{i\geq 1}$} and {\it $(\kappa_i)_{i\geq 1}$}  on $\mathbb{N}$ by
$$\mu_i(m)=\binom{a_i}{i-1}+\binom{a_{i-1}}{i-2}+\cdots+\binom{a_j}{j-1}$$
and 
$$\kappa_i(m)=\binom{a_i-1}{i-1}+\binom{a_{i-1}-1}{i-2}+\cdots+\binom{a_j-1}{j-1}.$$
The characterization of $f$- and $h$-vectors of simplicial complexes are due to Kruskal \cite{Kru63}, Katona \cite{Kat68} independently and Macaulay \cite{Maca27} respectively. Note that in combinatorial terms Macaulay's result characterizes $f$-vectors of multicomplexes which is equivalent to being the $h$-vector of a Cohen-Macaulay simplicial complex, see e.g. \cite{Sta} for 
the equivalence and the definition of Cohen-Macaulay simplicial complexes.

\begin{theorem}[Kruskal-Katona, Macaulay]\label{f-h}
Let $f=(1, f_0,\dots,f_{d-1})$ be a vector of positive integers. Then
\begin{itemize}
  \item $f$ is an $f$-vector of a simplicial complex if and only if $\mu_{i+1}(f_i)\leq f_{i-1}$ for $0\leq i\leq d-1$.
  \item $f$ is an $h$-vector of a Cohen-Macaulay simplicial complex if and only if $\kappa_{i+1}(f_i)\leq f_{i-1}$ for $0\leq i\leq d-1$.
\end{itemize}
 
\end{theorem}

As usual, for example in the contexts of Newton's generalized binomial
formula, we consider the binomial coefficient $\binom{n}{k}$ as
a polynomial function in $n$, so that for $x \in \mathbb{R}$ we have
$$\binom{x}{k}=\frac{x(x-1)\cdots (x-k+1)}{k!}.$$
We also write $(x)_k$ for the numerator $x(x-1) \cdots (x-k+1)$.
The following lemma recalls some well known facts about binomial coefficients.

\begin{lemma} \label{lem:facts}~
\begin{itemize}
    \item[(i)] For any $x \in \RR$ and any $0 < k \in \mathbb{N}$ we have
    $\binom{x}{k}+\binom{x}{k-1}=\binom{x+1}{k}$.
    \item[(ii)] For any $0 \leq y \in \mathbb{R}$ and any $0 < k \in \mathbb{N}$ there is 
    a unique $x \geq k-1$ such that $y = \binom{x}{k}$.
\end{itemize}
\end{lemma}
\begin{proof}
    Part (i) is a trivial calculation. The second assertion follows
    from the fact that $k-1$ is the largest root of the 
    $k$ real roots of $\binom{x}{k}$. Hence $\binom{x}{k}$ is strictly
    monotone for $x \geq k-1$ with limit $+ \infty$.
\end{proof}

By \ref{lem:facts} any polynomial $f(t)= 1 +y_1t+\cdots+y_dt^d$ with coefficients $y_i \geq 0$ there
are uniquely defined real numbers $x_i \geq i-1$ such that such that
$f(t)= 1+\binom{x_1}{1}t+\cdots+\binom{x_d}{d}t^{d}$. We call $(x_1,\ldots,x_d)$ the {\it binomial
expansion} of $f(t)$ or of its {\it coefficient sequence}. 

\section{Real-rooted polynomials with positive integer coefficients}

While \ref{qu:bs} remains open, we provide three different approaches to advancing our ability to answer it.

\subsection{First approach} 
We answer a special case of \ref{qu:bs} by describing a subset of real-rooted polynomials in $\mathbb{N}[t]$  which are provably $f$-polynomials of simplicial complexes.

Before we can formulate our first main results we recall the following definition. We say that a polynomial
$a_0 + a_1t+\cdots +a_dt^ d$ of degree $d$ is {\it ultra log-concave} if
$$\Big[ \,\frac{a_i}{\binom{d}{i}}\,\Big]^ 2
\geq \frac{a_{i-1} a_{i+1}}{\binom{d}{i-1}\binom{d}{i+1}}$$ for $1 \leq i \leq d-1$. This concept is motivated by Newton's result showing that
any real-rooted polynomial with non-negative coefficients has an 
ultra log-concave coefficient sequence (see for example \cite[Lemma 7.1.1]{Bra2}).

\begin{theorem}\label{xk-decreasing}
Let $f(t)=1+\binom{x_1}{1}t+\cdots+\binom{x_d}{d}t^{d}$ in $\mathbb{R}_{\geq 0}[t]$ with  $x_d \geq d$ be an ultra log-concave polynomial. 
Then  
$x_1\geq x_2\geq\cdots\geq x_{d}$.
\end{theorem}

\begin{proof}
By ultra log-concavity we have 
\begin{align}\label{NI}
\left[\frac{\binom{x_i}{i}}{\binom{d}{i}}\right]^2\geq \frac{\binom{x_{i-1}}{i-1}}{\binom{d}{i-1}}\frac{\binom{x_{i+1}}{i+1}}{\binom{d}{i+1}} \text{ for } 2 \leq i \leq d-1
\end{align}
and \begin{align} \label{N1}
\left[\frac{\binom{x_1}{1}}{\binom{d}{1}}\right]^2 \geq \frac{\binom{x_2}{2}}{\binom{d}{2}}.
\end{align}

\noindent {\sf Claim:} For integers $1 \leq i \leq d-1$ we have
$$\sqrt[i]{\frac{\binom{x_i}{i}}{\binom{d}{i}}} \geq
\sqrt[i+1]{\frac{\binom{x_i}{i}}{\binom{d}{i}}}.$$

We prove the claim be induction on $i$. For $i =1$ the assertion follows
immediately from \eqref{N1}. Now let $i > 1$.
Using \eqref{NI} and the induction hypothesis for $i-1$ we get
\begin{align*}
\left[\frac{\binom{x_i}{i}}{\binom{d}{i}}\right]^2 \overset{\text{\eqref{NI}}}{\geq} 
\frac{\binom{x_{i-1}}{i-1}}{\binom{d}{i-1}}\frac{\binom{x_{i+1}}{i+1}}{\binom{d}{i+1}} \overset{\text{Induction}}{\geq} \left[\frac{\binom{x_{i}}{i}}{\binom{d}{i}}\right]^{\frac{i-1}{i}}\frac{\binom{x_{i+1}}{i+1}}{\binom{d}{i+1}}.
\end{align*}
Now cancelling $\frac{\binom{x_i}{i}}{\binom{d}{i}}$ in the left and right terms and taking the $(i-1)$\textsuperscript{st} root yields the claim.

The claim immediately implies the following inequalities
$$\frac{\binom{x_1}{1}}{\binom{d}{1}} \geq \sqrt{\frac{\binom{x_2}{2}}{\binom{d}{2}}}\geq\cdots\geq \sqrt[i]{\frac{\binom{x_i}{i}}{\binom{d}{i}}}\geq\cdots\geq \sqrt[d]{\binom{x_d}{d}}.$$
Then
$$\binom{x_i}{i}\geq \left[ \binom{x_d}{d}\right]^{i/d}\binom{d}{i}.$$
Since $x_d \geq d$ we have $\left[ \binom{x_d}{d}\right]^{i/d} \geq 1$
and hence $\binom{x_i}{i}\geq \binom{d}{i}$ and $x_i \geq d$ for all $i$.

The proof of the theorem is now by induction on $i$.
For the case $i=1$, by \eqref{N1}, we have $(d-1)x_1^2\geq dx_2(x_2-1)$.
Assume $x_1< x_2$ then
$$(d-1)x_1^2<(d-1)x_2^2\leq dx_2(x_2-1),$$
which contradicts \eqref{N1} and $x_1\geq x_2$ follows.

For the induction step let $i \geq 2$ and assume that $x_{i}\leq x_{i-1}$.
By \eqref{NI},
$$\frac{(x_i)_i^2}{(d)_i^2}\geq \frac{(x_{i-1})_{i-1}}{(d)_{i-1}}\frac{(x_{i+1})_{i+1}}{(d)_{i+1}}.$$
By the induction hypothesis,
$$\frac{(x_i)_i^2}{(d)_i^2}\geq \frac{(x_{i})_{i-1}}{(d)_{i-1}}\frac{(x_{i+1})_{i+1}}{(d)_{i+1}},$$
which equivalent to
\begin{align}\label{xi}
(d-i)\cdot (x_i)_i\cdot (x_i-i+1)\geq (d-i+1)\cdot (x_{i+1})_{i+1}.
\end{align}
Assume that $x_i<x_{i+1}$, then
$$(d-i)\cdot (x_i)_i\cdot (x_i-i+1)< (d-i)\cdot (x_{i+1})_i\cdot (x_{i+1}-i+1)\leq (d-i+1)\cdot (x_{i+1})_{i+1},$$
where the last inequality follows from $d\leq x_{i+1}$.
This contradicts \eqref{xi} and the proof is complete.
\end{proof}

By \cite[Example 2.3]{BH} and classical results by Newton the assumption of 
\ref{xk-decreasing} is satisfied in the following two cases.

\begin{corollary} Let $f(t)=1+\binom{x_1}{1}t+\cdots+\binom{x_d}{d}t^{d}$ in $\mathbb{R}_{\geq 0}[t]$ such that $x_d \geq d$.
Then the following holds:
    \begin{itemize}
        \item[(i)]  If $f(t)$ is real-rooted then 
$x_1\geq x_2\geq\cdots\geq x_{d}$.
        \item[(ii)] If the polynomial $ s^d+\binom{x_1}{1}\,ts^ {d-1} +\binom{x_2}{2}\,t^2s^ {d-2} +\cdots+
       \binom{x_d}{d}\,t^{d}$ is Lorentzian then 
$x_1\geq x_2\geq\cdots\geq x_{d}$.

    \end{itemize}
\end{corollary}

An immediate consequence of \ref{xk-decreasing} is the following result which is the main result in \cite{BS07}, where it is formulated in terms of
$f$-vectors of multicomplexes.

\begin{corollary} \label{cor:hvector}
Let $f(t)=1+f_{0}t+\cdots +f_{d-1}t^{d}\in \mathbb{N}[t]$ be a polynomial with only real roots and
$f_i=\binom{x_{i+1}}{i+1}$. Then $f(t)$ is the $h$-polynomial of a Cohen-Macaulay simplicial complex.
\end{corollary}

\begin{proof}
Define $x_1, \dots,x_d$ by writing $f_i=\binom{x_{i+1}}{i+1}$ for $i=0,\dots,d-1$. Now fix an index $i$ and let the $(i+1)$st binomial expansion of $f_i$ be 
$$\binom{a_{i+1}}{i+1}+\binom{a_{i}}{i}+\cdots\binom{a_j}{j}.$$
By \ref{cor:hvector}, we have \begin{align} \label{eq:1} a_{i+1} & \leq x_{i+1}\leq x_{i}. \end{align} 

It follows that
$$\kappa_{i+1}(f_i)=\binom{a_{i+1}-1}{i}+\binom{a_{i}-1}{i-1}+\cdots+\binom{a_j-1}{j-1}\overset{\text{\eqref{eq:binomial}}}{\leq} \binom{a_{i+1}}{i}\overset{\text{\eqref{eq:1}}}{\leq} \binom{x_{i}}{i}=f_{i-1}.$$
By \ref{f-h}, $f(t)$ is the $h$-polynomial of a simplicial complex.
\end{proof}

Now we are position to formulate and prove our main result 
which provides a positive answer to \ref{qu:bs} in a special case.

\begin{theorem} \label{thm:main}
Let $f(t) = 1+f_0t+\cdots +f_{d-1}t^ d \in \mathbb{N}[t]$ be a
polynomial of degree $d$ with only real roots and 
let $f(t)=1+\binom{x_1}{1}t+\cdots+\binom{x_d}{d}t^{d}\in \mathbb{N}[t]$ be its binomial expansion.
If $x_{i-1}\geq \lceil x_i\rceil$ for $i=2,\dots,d$ then $f(t)$ is the $f$-polynomial of a simplicial complex.
\end{theorem}

\begin{proof}
For a fixed $i$ in $\{1,\ldots, d\}$, let 
$$f_{i-1} =\binom{a_i}{i}+\binom{a_{i-1}}{i-1}+\cdots+\binom{a_j}{j},$$
with $a_i > \cdots > a_j \geq j \geq $ be the $i$\textsuperscript{th} binomial expansion of $f_{i-1}$.
By the monotonicity of $\binom{x}{i}$ it follows that $a_i\leq x_i< a_i+1$. 
If $a_i= x_i$ then we have $f_{i-1}=\binom{x_i}{i}=\binom{a_i}{i}$
and
$$\mu_i(f_{i-1})=\binom{x_i}{i-1}=\binom{\lceil x_i\rceil}{i-1}\leq \binom{x_{i-1}}{i-1}=f_{i-2}.$$
If $a_i< x_i<a_i+1$ then $a_i+1=\lceil x_i\rceil \leq x_{i-1}$.
Hence $$\mu_i(f_{i-1})=\binom{a_i}{i-1}+\binom{a_{i-1}}{i-2}+\cdots+\binom{a_j}{j-1}\overset{\text{\eqref{eq:binomial}}}{\leq} \binom{a_i+1}{i-1}\leq \binom{x_{i-1}}{i-1}=f_{i-2}.$$
By \ref{f-h}, $f(t)$ is the $f$-polynomial of a simplicial complex.
\end{proof}

For the Eulerian polynomial, the Stirling polynomial of the
second kind, and the derangement polynomial we have computed
the binomial expansion for first $200$ polynomials in each series. 
In each case the data satisfies the assumptions of \ref{thm:main}
(see \ref{fig:eu} for some initial values in the case of Eulerian polynomials).

\begin{figure}
\begin{tabular}{c|ccccccc}
$n$ & $x_1$ & $x_2$ & $x_3$ & $x_4$ & $x_5$ & $x_6$ & $x_7$\\
\hline
$2$ & $1$ & & & & & & \\
$3$ & $4$ &  $2$ \\
$4$ & $11$ & $5.216\cdots$ &  $3$ \\
$5$ & $26$ & $12$ & $6.445 \cdots$ &  $4$ \\
$6$ & $57$ &  $25.081\cdots$  &  $13.218\cdots$ &  $7.684\cdots$ &  $5$ \\
$7$ & $120$ & $49.308 \cdots$ &  $25.396\cdots$ &  $14.550\cdots$ &  $8.934\cdots$ & $6$ \\
$8$ & $247$ &  $93.162\cdots$ &  $46.429\cdots$ &  $26.269\cdots$ &  $15.952\cdots$ &  $10.194\cdots$ &  $7$
\end{tabular}
\caption{Binomial expansion for degree $n$ Eulerian polynomials}\label{fig:eu}
\end{figure}

Nevertheless, we were not able to verify the assumptions
of \ref{thm:main} for any of these sequences. Instead, we
were able to cover them in  \ref{cor:euler} using our third
approach to \ref{qu:bs}.

\begin{remark}
Note that the assumptions $x_{i-1}\geq \lceil x_i\rceil$ for $i=2,\dots,d$ of \ref{thm:main} is not satisfied by every real-rooted polynomial $1+\binom{x_1}{1}t+\cdots+\binom{x_d}{d}t^{d}$ in $\mathbb{N}[t]$ . 
For example, the polynomial $f(t)=1+4t+5t^2+2t^3 = (x+1)^ 2(2x+1)$
has only real roots,
but when expressed as $f(t) = 1+\binom{x_1}{1}t + \binom{x_2}{2} +
\binom{x_3}{3}t^ 3$
then $x_1 = 4$, $x_2 = \frac{1}{2} (1+\sqrt{41}) = 3.701562118 \cdots$
and $x_3 = \frac{1}{3}\,\sqrt [3]{162+3\,\sqrt {2913}}+{\frac {1}{\sqrt [3]{162+3\,\sqrt 
{2913}}}}+1 = 3.434841368\cdots$.
It follows that $$x_{2} = 3.701562118 \not\geq 4 = \lceil 3.434841368\cdots \rceil = \lceil x_3 \rceil.$$ 
On the other hand $(1,4,5,2)$ is the $f$-vector of the simplicial complex
given by two triangles glued along an edge.
\end{remark}

\begin{remark}
Note that the assumption $x_{i-1}\geq \lceil x_i\rceil$ of \ref{thm:main} is sometimes satisfied by $f$-polynomials of simplicial complexes having imaginary roots. For example,
$f(t)= 1+4t+6t^2+3t^3 = 1+ \binom{4}{1}t + \binom{4}{2}t^2+\binom{3.747836837 \cdots}{3}t^3$
is not real-rooted. It is the $f$-polynomial of the boundary of the tetrahedron with 
the interior of one triangle removed. On the other hand $x_{i-1} \geq \lceil x_{i} \rceil$
is still satisfied.
\end{remark}

\begin{remark}
The following example shows that if 
the binomial expansion $(x_1,\ldots, x_d)$ of the
$f$-vector of a simplicial complex satisfies $x_{i-1}\geq \lceil x_i\rceil$ for $i=2,\ldots, d$ then the corresponding $f$-polynomial 
$f(t)$ is not necessarily real-rooted, indeed not even log-concave.
The polynomial
$f(t) =1+10t+3t^2+t^3=1+\binom{10}{1}t+\binom{3}{2}t^2+\binom{3}{3}t^3$ is the $f$-polynomial of
a triangle together with $7$ isolated vertices. Clearly, $10 \geq 3\geq 3$ but
$9 = 3^2 \not\geq 1 \cdot 10 = 10$.
\end{remark}


\subsection{Second approach} 
We further advance our understanding of \ref{qu:bs} by conjecturing a stronger result which implies an affirmative answer to the question.
Let 
$f(t)=1+f_{0}t+\cdots +f_{d-1}t^{d}\in \mathbb{N}[t]$  and let
$f(t)=1+\binom{x_1}{1}t+\cdots+\binom{x_d}{d}t^{d}\in \mathbb{N}[t]$ be its binomial expansion. Further let 
$$f_{i-1} = \binom{a_i}{i}+\binom{a_{i-1}}{i-1}+\cdots+\binom{a_j}{j}, \text{ for } a_i > \cdots > a_j \geq j \geq 1$$ be the $i$\textsuperscript{th} binomial expansion of $f_{i-1}$.
For $i \geq 1$, we define $g_i$ and $h_{i-1}$ by
 $$g_i =\binom{a_i-1}{i}+\binom{a_{i-1}-1}{i-1}+\cdots+\binom{a_j-1}{j}$$
and $$h_{i-1} =\binom{a_i-1}{i-1}+\binom{a_{i-1}-1}{i-2}+\cdots+\binom{a_j-1}{j-1}.$$
It is easily seen that $h_0 = 1$ and $f_{i-1} = g_i + h_{i-1}$. In particular, we have 
$$ f(t) = 1+\sum_{i=1}^ d  f_{i-1}t^ i = \Big(\,1 + \sum_{i=1}^d 
g_i t^ i \,\Big)+ t \Big(\,1 + \sum_{i=2}^ d h_{i-1} t^{i-1}\,\Big).$$
We call the decomposition $f(t) = g(t)+th(t)$ with
$g(t) =1+\sum_{i=1}^d  g_it^ i$ and
$h(t) = 1+\sum_{i=1}^{d-1}  h_it^ i$ the {\it recursive decomposition}
of $f(t)$.

We will show that a positive answer to the following conjecture implies a positive answer to \ref{qu:bs}.


\begin{conjecture} \label{con:second}
Let $f(t)=1+f_{0}t+\cdots +f_{d-1}t^{d}\in \mathbb{N}[t]$ have only real roots and recursive decomposition
 $f(t) = g(t)+th(t)$, with 
 $$g(t) = 1+g_1t+\cdots+g_dt^d,\ \  h(t) = 1+h_1t+\cdots+h_{d-1}t^{d-1}.$$
 Then for $i=1,\dots,d-1$, we have
 $h_i\leq g_i$.
\end{conjecture}

\begin{proposition} If \ref{con:second} holds then 
\ref{qu:bs} has a positive answer.
\end{proposition}
\begin{proof}
Let 
$$f_{i-1} = \binom{a_i}{i}+\binom{a_{i-1}}{i-1}+\cdots+\binom{a_j}{j}, \text{ for }  a_i > \cdots a_j \geq j \geq 1,$$ and 
$$f_{i} = \binom{b_{i+1}}{i+1}+\binom{b_{i}}{i}+\cdots+\binom{b_r}{r}, \text{ for } b_{i+1} > \cdots > b_r \geq r \geq 1$$
be the $i$\textsuperscript{th} binomial expansion of $f_{i-1}$ and the $(i+1)$\textsuperscript{st} binomial expansion of $f_i$.
Then  $$g_i=\binom{a_i-1}{i}+\binom{a_{i-1}-1}{i-1}+\cdots+\binom{a_j-1}{j}$$
and $$h_i=\binom{b_{i+1}-1}{i}+\binom{b_{i}-1}{i}+\cdots+\binom{b_r-1}{r-1}.$$
After canceling identical terms in the binomial expansions and using
\eqref{eq:binomial} it follows that $h_i\leq g_i$ if and only if 
 $(b_{i+1},b_{i},\dots,b_r)\leq_{lex} (a_i,a_{i-1},\dots,a_j)$ in the lexicographic order.
 Note that $$\mu_{i+1}(f_{i}) = \binom{b_{i+1}}{i}+\binom{b_{i}}{i-1}+\cdots+\binom{b_r}{r-1}.$$
 If  $(b_{i+1},b_{i},\dots,b_r)\leq_{lex} (a_i,a_{i-1},\dots,a_j)$ then we have $\mu_{i+1}(f_{i})\leq f_{i-1}$.
 The assertion now follows.

\end{proof}

Some calculations on the Eulerian polynomial and the Stirling polynomial of the second kind suggest a positive answer, upon which we pose the following question.

\begin{question}\label{que:second}
    Let $f(t)=1+f_{0}t+\cdots +f_{d-1}t^{d}\in \mathbb{N}[t]$ be a polynomial with only real roots
and let $f(t) = g(t)+th(t)$ its recursive decomposition.
Is true that
 $g(t)$ and $h(t)$ both have only real roots?
\end{question} 

After this paper was published, in a parallel development, Anish Patil and
the authors found examples showing that the answer to the question is
negative in general. The example provided to us by Anish Patil is 
particularly simple. Let $f(t) = (1 + 3t)^ 3$ then the recursive  decomposition is given by $g(t) = (1 + t)(1 + 7t + 13t^ 2)$ and $
h(t) = 1 + 7t + 14t^ 2$. While $f(t)$ clearly has only real roots, the
quadratic polynomials $1 + 7t + 13t^2$ and $1 + 7t + 14t^ 2$ have no
real roots. After being informed by Anish Patil about this counterexample
the paper Feihu Liu, Ying Wang, Zihao Zhang posted a paper \cite{LWZ} 
with results and counterexamples completely clarifying the boundaries between 
yes and no answers to the question.

\subsection{Third approach} 
To affirmatively answer another special case of \ref{qu:bs}, we consider certain vectors studied by Murai \cite{Mur10}.
Let $f=(1, f_0,\dots,f_{d-1})\in \mathbb{Z}^{d+1}$ be the $f$-vector of a simplicial complex.
We say that a vector $\alpha=(0,1, \alpha_0,\dots,\alpha_{d-2})\in \mathbb{Z}^{d+1}$ is a {\it basic admissible vector} of $f$ if
\begin{itemize}
  \item[(BA1)] $(1, \alpha_0,\dots,\alpha_{d-2}) \in \mathbb{Z}^{d}$ is the $f$-vector of a simplicial complex;
  \item[(BA2)] $f_i\geq \alpha_i$ for $i=0,1,\dots, d-2$.
\end{itemize}
A vector $\beta\in \mathbb{Z}^{d+1}$ is {\it admissible} to $f$ if there exists a sequence of vectors $\beta_1,\dots,\beta_t\in\mathbb{Z}^{d+1}$
such that \begin{itemize}
            \item[(A1)] $\beta=\beta_1+\cdots+\beta_t$;
            \item[(A2)] $\beta_i$ is a  basic admissible vector of $f+\beta_1+\cdots+\beta_{i-1}$ for $i=1,2,\dots,t$.
          \end{itemize}

\begin{lemma}[Lemma 3.1 (1)(3)\cite{Mur10}]\label{admissible}
Let $f=(1, f_0,\dots,f_{d-1})\in \mathbb{Z}^{d+1}$ be the $f$-vector of a simplicial complex
and let $\alpha=(0,1, \alpha_0,\dots,\alpha_{d-2}), \beta=(0,1,\beta_0,\ldots, \beta_{d-2}) \in \mathbb{Z}^{d+1}$ be admissible to $f$. Then 
\begin{itemize}
    \item [(i)] $f+\alpha$ is the $f$-vector of a simplicial complex.
    \item[(ii)] $\alpha+\beta$ is admissible to $f$.
    \end{itemize}
\end{lemma}

The next proposition will be key to the proof that some real-rooted classical polynomials are $f$-polynomials. 

\begin{proposition}\label{thm:f+f'}
Let $f(t)$ be an $f$-polynomial of a simplicial complex. Then $f(t)+tf'(t)$ is an $f$-polynomial of a simplicial complex.
\end{proposition}

\begin{proof}
Suppose that $f(t)=1+f_{0}t+\cdots +f_{d-1}t^{d}$ is an $f$-polynomial of a simplicial complex $\Delta$.
Then $f=(1, f_0,\dots,f_{d-1})\in \mathbb{Z}^{d+1}$ is the $f$-vector of $\Delta$.
Let $\beta=(0,f_0, 2f_1,\dots,df_{d-1})\in \mathbb{Z}^{d+1}$.
Then by \ref{admissible}, to show that $f(t)+tf'(t)$ is an $f$-polynomial of a simplicial complex it suffices to show
 $\beta$ is an admissible vector of $f$.
For each vertex $v_i$, $i=1,\ldots, f_0$, of 
$\Delta$ let
$\beta_i = (0,1,\beta_0^{(i)},\ldots, \beta_{d-2}^{(i)})$, where
$(1,\beta_0^{(i)},\ldots, \beta_{d-2}^{(i)})$ is
the $f$-vector of the link of $v_i$ in $\Delta$. Then (BA1) is satisfied for all
$\beta_i$ by definition. Since the $(i-1)$-faces
of the link of a vertex are in one to one correspondence with a subset of the $i$-faces
of $\Delta$, condition (BA2) follows for all
$\beta_i$. Indeed it follows for all
$\beta_i$ and for any complex
having an $f$-vector which is coordinatewise 
greater or equal to $(1, f_0,\dots,f_{d-1})$.
Using this and \ref{admissible}(i),(ii) inductively 
it then follows that 
$f + \beta_1 + \cdots + \beta_{i-1}$ is a basic
admissible vector of $f+\beta_1 + \cdots + 
\beta_i$ for $i=1,\ldots, f_0$. 
In particular, \ref{admissible} implies that
$f+\beta$ is an $f$-vector.

Since a single $i$-face of $\Delta$ gives rise to exactly $i+1$ faces of dimension $i-1$ in the
links of its vertices it follows that
$\beta = (0,f_0,2f_1,\ldots, df_{d-1})$
and the assertion follows.
\end{proof}

The following corollary will turn out to be very useful when analyzing sequences defined through 
recursions.

 \begin{corollary}\label{f-vector}
 Let $(1,f_0,\dots,f_{d-1})$ be the $f$-vector of a simplicial complex. Then $$\mu_{i+1}\big(\,(i+2)f_i\,\big)\leq (i+1)\,f_{i-1} \text{ for } 1\leq i\leq d-1.$$
 \end{corollary}
\begin{proof}
Basic calculus shows that for
$f(t) = 1+\sum_{i=1}^{d}f_{i-1}t^i$ we
have that $f(t)+tf'(t) = 
1+2f_0t+\cdots +(d+1)f_{d-1}$. By  \ref{thm:f+f'} it follows that $(1,2f_0,3f_1,\dots,(i+2)f_i,\dots,(d+1)f_{d-1})$ is the $f$-vector of a simplicial complex.
Then by \ref{f-h}, the result follows.
\end{proof}

We now want to apply these results to show that
Eulerian, Stirling and derangement polynomials are $f$-polynomials. Instead of giving three separate proofs, we consider arrays
$\big[\,T_{d,k}\,\big]_{d,k \in \mathbb{Z}}$ of integers which satisfy the following
conditions:

\begin{itemize}
    \item[(R1)] $T_{1,0}=1$,
    \item[(R2)] $T_{d,k}=0$ for all $d$ and $k$ which do not satisfy $0\leq k\leq d-1$,
    \item[(R3)] For $i=1,2,3$ there are integers $r^{(i)},s^{(i)},t^{(i)}$ such that
    \begin{itemize}
        \item[(i)] $r^{(1)} = 0$ and $t^ {(1)} =1$,
        \item[(ii)] $s^ {(i)} \leq t^{(i)}$, 
        \item[(iii)]  for all $0 \leq k \leq d-1$ and $d \geq 2$ we have $a_{d,k}^{(i)} :=r^{(i)}d+s^{(i)}k+t^{(i)} \geq 0$.
        \end{itemize}
    \item[(R4)] For $0 \leq k \leq d-1$
    and $d \geq 2$ we have
\begin{align*}\label{triangle}
T_{d,k}=a_{d,k}^{(1)}T_{d-1,k}+a_{d,k}^{(2)}T_{d-1,k-1}+a_{d,k}^{(3)}T_{d-2,k-1}.
\end{align*}
\end{itemize}

For the following choices of parameters we recover
Eulerian, Stirling and derangement numbers
and verify that (R1)-(R4) are satisfied.

\begin{itemize}
  \item {\sf (Eulerian)}  $[T_{d,k}]_{d\geq 1, k\geq 0}$ is the Eulerian triangle,
  with
  $T_{d,k}$ the number of permutations in $S_d$ with $k$ descents (see \cite[A008292]{Slo}) for $a_{d,k}^{(1)}=k+1
  = 0 \cdot d+1\cdot k + 1$, $a_{d,k}^{(2)}=
  d-k = 1 \cdot d + (-1) \cdot k + 0$ and $a_{d,k}^{(3)}=0 = 0 \cdot d + 0 \cdot k + 0$.
  \item {\sf (Stirling)} $[T_{d,k}]_{d\geq 1, k\geq 0}$ is the Stirling triangle of the second kind,
  with $T(d,k)$ the number of set partitions of
  $\{1,\ldots, d\}$ with $k+1$ blocks (see \cite[A008277]{Slo}) for $a_{d,k}^{(1)}=k+1 = 0 \cdot d+1\cdot k + 1$, $a_{d,k}^{(2)}=1= 0 \cdot d + 0 \cdot k + 1$ and $a_{d,k}^{(3)}=0 =0 \cdot d + 0 \cdot k + 0$.
  \item {\sf (Derangements)} $T_{d,k}$ is the number of derangements of $[d]$ having $k$ exceedances (see \cite[A046739]{Slo}) for $a_{d,k}^{(1)}=k+1= 0 \cdot d+1\cdot k + 1$, $a_{d,k}^{(2)}=d-k= 1 \cdot d + (-1) \cdot k + 0$ and $a_{d,k}^{(3)}=1 \cdot d+0 \cdot k +0$.
\end{itemize}

\begin{theorem} \label{thm:T-poly}
Let $[T_{d,k}]_{d,k \in \mathbb{Z}}$ be a triangular array defined satisfying (R1), (R2), (R3) and (R4).
Then the $d$\textsuperscript{th} row polynomial $T_d(t)=\displaystyle{\sum_{k\geq 0}}T_{d,k}t^k$ is an $f$-polynomial of a simplicial complex.
\end{theorem}

\begin{proof}
Note that by (R2) and (R4) $T_{d,0}= a_{d,0}^{(1)} T_{d-1,0}$ for $d \geq 2$. Now (R1) and (R3)(i) imply $T_{d,0} = 1$ for all $d\geq 1$.

We prove the assertion by induction on $d$. 

For $d=1$ condition (R1) and (R2) imply that 
$T_1(t) = 1$ is the $f$-polynomial of the simplicial complex $\{ \emptyset \}$. 

For $d=2$ we have by (R1), (R2) and (R3)(iii) that $T_{2,1} = a_{2,1}^{(2)} \geq 0$. Furthermore, $(1,a_{2,1}^ {(2)})$ is the $f$-polynomial
of the simplicial complex consisting of $a_{2,1}^ {(2)}$
simplices of dimension $0$ and the empty set. 

Suppose that $d \geq 3$ and for $n \leq d-1$ 
all $T_n(t)$ are $f$-polynomials 
of simplicial complexes.
By \ref{f-h} we need to show that for $1 \leq k \leq d-1$
we have $\mu_k(T_{d,k}) \leq T_{d,k-1}$.

We will repeatedly use the fact 
that $\mu_k$ is a subadditive function on
non-negative integers (i.e., $\mu_k(a+b) \leq \mu_k(a) + \mu_k(b)$). For that we first extend
$\mu_k$ to all
non-negative integers by setting $\mu_k(0) = 0$. It is easily checked that the  subadditivity of $\mu_k$ on positive integers (see \cite[Theorem 7.4.2]{And} or \cite{Cle74}) extends. We mark arguments using the subadditivity by $(sub)$.

\medskip

\noindent {\sf Claim:} 
For $d \geq 3$ and $0 \leq k \leq d-1$ we have 
\begin{itemize}
    \item[(1)] $\mu_k\left(a_{d,k}^{(1)}T_{d-1,k}\right)
\leq a_{d,k-1}^{(1)}T_{d-1,k-1}$,
    \item[(2)] $\mu_k\left(a_{d,k}^{(2)}T_{d-1,k-1}\right)
\leq a_{d,k-1}^{(2)}T_{d-1,k-2}$,
 \item[(3)] $\mu_k\left(a_{d,k}^{(3)}T_{d-2,k-1}\right)
\leq a_{d,k-1}^{(3)}T_{d-2,k-2}$.
\end{itemize}

\medskip

We prove the inequality (1) of the claim. The others then follow by similar reasoning.

\noindent {\sf Case:} $s^{(1)}\leq 0$.

By the induction hypothesis for $d-1$ and since by
 (R3)(iii) $a_{d,k}^{(1)}\geq 0$, we have
 \begin{align*}
\mu_k\left(a_{d,k}^{(1)}T_{d-1,k}\right)& \overset{\text{(R3)}}{=}
\mu_k\left((r^{(1)}d+s^{(1)}k+t^{(1)})\,T_{d-1,k}\right)\\
&\overset{(sub)}{\leq} (r^{(1)}d+t^{(1)})\,\mu_k(T_{d-1,k})+s^{(1)}\mu_k(kT_{d-1,k})\\
&\overset{\text{\ref{f-vector}}}{\leq} (r^{(1)}d+t^{(1)})
\,\mu_k(T_{d-1,k}) + s^{(1)}(k-1)\,T_{d-1,k-1}\\
&\overset{\text{Induction}}{\leq} (r^{(1)}d+s^{(1)}(k-1)+t^{(1)})\,T_{d-1,k-1}\\
&\overset{\text{(R3)}}{=} a_{d,k-1}^{(1)}T_{d-1,k-1}.
\end{align*}

\noindent {\sf Case:} $s^{(1)}>0$.

By $s^{(1)}\leq t^{(1)}$ and $a_{d,k}^{(1)}\geq 0$,
\begin{align*}
\mu_k\left(a_{d,k}^{(1)}\,T_{d-1,k}\right)&\overset{\text{(R3)}}{=}
\mu_k\left((r^{(1)}d+s^{(1)}k+t^{(1)})\,T_{d-1,k}\right)\\
&=\mu_k \left((r^{(1)}d+s^{(1)}(k+1)+t^{(1)}-s^{(1)})\,T_{d-1,k}\right)\\
&\overset{(sub)}{\leq} \mu_k \left(s^{(1)}(k+1)\,T_{d-1,k}\right)+ \mu_k \left((r^{(1)}d+t^{(1)}-s^{(1)})\,T_{d-1,k}\right) \\
&\overset{\text{\ref{f-vector}}}{\leq} s^{(1)}k\,T_{d-1,k-1}+
\mu_k \left((r^{(1)}d+t^{(1)}-s^{(1)})\,T_{d-1,k}\right)\\
&\overset{\text{Induction}}{\leq} (r^{(1)}d+s^{(1)}k+t^{(1)}-s^{(1)})\,T_{d-1,k-1}\\
&\overset{(R3)}{=} a_{d,k-1}^{(1)}T_{d-1,k-1},
\end{align*}
 This completes the proof of the claim. 
Using (R4) and (R3)(iii) we conclude
\begin{align*}
\mu_k(T_{d,k})& \overset{\text{(R4)}}{=}\mu_k\left(a_{d,k}^{(1)}T_{d-1,k}+a_{d,k}^{(2)}T_{d-1,k-1}+a_{d,k}^{(3)}T_{d-2,k-1}\right)\\
&\overset{(sub)}{\leq} \mu_k\left(a_{d,k}^{(1)}T_{d-1,k}\right)+\mu_k\left(a_{d,k}^{(2)}T_{d-1,k-1}\right)+\mu_k\left(a_{d,k}^{(3)}T_{d-2,k-1}\right) \\
&\overset{\text{(Claim)}}{\leq} a_{d,k-1}^{(1)}T_{d-1,k-1}+a_{d,k-1}^{(2)}T_{d-1,k-2}+a_{d,k-1}^{(3)}T_{d-2,k-2}\\
&\overset{\text{(R4)}}{=} T_{d,k-1}.
\end{align*}

The assertion now
follows.
\end{proof}

Since we have already verified conditions
(R1)-(R4) for these classical polynomials 
the following corollary is an immediate consequence of \ref{thm:T-poly}.
The case of the Eulerian polynomial is
known by work of Edelman and Reiner \cite{ER94} and Gasharov \cite{Gash98}. For the other two we are not aware of any reference.

\begin{corollary} \label{cor:euler}
The Eulerian polynomial, the Stirling polynomial of the second kind and the derangement polynomial are $f$-polynomials of simplicial complexes.
\end{corollary}

\section{Preserving real-rootedness preserves the \texorpdfstring{$f-$}~vector property} \label{sec:construction}

In this section we collect results which show that well studied
constructions on polynomials, that preserve real-rootedness, also preserve
the property of being the $f$-polynomial of a simplicial 
complex. Since for a real-rooted polynomial $f(t)$ the polynomial
$f(t) + f'(t)\,t$ is also real-rooted, \ref{thm:f+f'} is another
result in this direction.

We start with the dilation of the variable. Elementary analysis shows that if $f(t)$ is a polynomial
with only real roots then same holds for $f(ct)$ for
any $c \in \RR$.

\begin{proposition}
Let $f(t)$ be an $f$-polynomial of a simplicial complex. Then for any positive $c \in \mathbb{Z}$ the polynomial $f(ct)$ is an $f$-polynomial of a simplicial complex.
\end{proposition}

\begin{proof}
Let $\Delta$ be a simplicial complex on 
vertex set $[d]$. We construct a simplicial
complex $\Delta^c$ on vertex set $[d]\times [c]$.
Define $\pi$: $[d]\times [c]\rightarrow [d]$ by $\pi\big(\,(i,j)\,\big)=i$.
For $\sigma \subseteq [d]\times [c]$ we set 
$\sigma \in \Delta^c$ if and only if 
$\pi(\sigma) \in \Delta$. Thus for any $i$-simplex $\sigma\in \Delta$ there are $c^{i+1}$ simplices of dimension $i$
in $\Delta^ c$.  
It follows that if $f(x)$ is the $f$-polynomial of $\Delta$ then $f(cx)$ is the $f$-polynomial of $\Delta^c$.
\end{proof}

It is easily seen that if $(a_n)_{n \geq 0}$ is
a P\'olya-frequency sequence, then so 
is $(a_{kn})_{n \geq 0}$ for any natural number $k \geq 1$
(see e.g. \cite{Br89}). In particular by the Aissen-Edrei-Schoenberg-Whitney Theorem, if 
$f(t) = 1+f_0t+ \cdots +f_{d-1}t^{d}$ is real-rooted
then so is $1+f_{k-1}t+f_{2k-1}t^ 2+\cdots$. 
For proving the following result we will resort
to methods from combinatorial commutative algebra 
and the theory of Gr\"obner bases. In the proof
we will use these methods freely and 
refer the reader to \cite{HH} for further
details.

\begin{proposition}
Let $f=(1, f_0,\ldots,f_{d-1})$ be the $f$-vector of a simplicial complex.
Then for any $k \geq 1$ the vector $(1, f_{k-1},f_{2k-1},\ldots)$ is the $f$-vector of a simplicial complex.
\end{proposition}
\begin{proof}
For this result we use the fact that $(1,f_0,\ldots, f_{d-1})$ is the $f$-vector of a simplicial complex if and only
if for any field $\KK$ there is a number $n$ and a squarefree monomial 
ideal $I$ in $S = \KK[x_1,\ldots, x_n]$ such that
the Hilbert series of $R= S/\big(\,I+(x_1^ 2,\ldots, x_n^ 2)\,\big)$ 
is $1+f_0t +\cdots +f_{d-1}t^ {d}$.  
In the latter case we have a direct sum decomposition
as $\KK$-vector spaces
$R = A_0\oplus A_1\oplus \cdots \oplus A_{d-1}$ where $A_i$ is the subspace generated by the 
cosets of the homogeneous polynomials of degree $i$
and is of $\KK$-dimension $f_{i-1}$. 
Now $R^{\langle k \rangle} = A_0 \oplus A_k
\oplus A_{2k} \oplus \cdots$ is known as the
$k$\textsuperscript{th} Veronese algebra of $A$. 
We can write $R^{\langle k \rangle}$ as the 
image
of $$T = \KK\big[\,y_{\ell_1,\ldots,\ell_n}~\big|~\ell_1+\cdots +
\ell_n = k \text{ and } \ell_i \geq 1\, \big]$$
by the homomorphism $\phi$ which sends $y_{\ell_1,\ldots,\ell_n}$ to $x_1^{\ell_1}\cdots x_n^{\ell_n}$. 
Since $x_i^2 = 0$ in $R$ it follows that
all $y_{\ell_1,\ldots,\ell_n}^2$ lie in the kernel of 
$\phi$. Let $M$ be the ideal generated by
the $y_{\ell_1,\ldots,\ell_n}^2$ and
$J$ be such that $J+M = \ker(\phi)$. 
Let $\preceq$ be a term order. Then 
$T/(M+J)$ and $T/{\mathrm in}_\preceq{(M+J)}$ share
the same Hilbert series, where ${\mathrm in}_\preceq(M+J)$
denotes the initial ideal of $M+J$.
Since $M$ is a monomial ideal we clearly have
$M \subseteq {\mathrm in}_\preceq(M+J)$. Since ${\mathrm in}_\preceq(M+J)$ also is a monomial ideal it follows
that ${\mathrm in}_\preceq(J+M) = J' + M$ for a 
squarefree monomial ideal $J'$. 
From this we deduce that 
$(1,f_{k-1},f_{2k-1},\ldots)$ is an $f$-vector.
\end{proof}

The fact that products of real-rooted polynomials are real-rooted is trivial. Also
the first part of the following proposition, which is the $f$-vector analog of this statement is well known.
The real-rootedness of the Hadamard product
of two real-rooted polynomials is again a known fact (see e.g. \cite{GW96}). 

\begin{proposition}
Let $f(t)=1+f_{0}t+\cdots +f_{d-1}t^{d}$ and $g(t)=1+g_{0}t+\cdots +g_{d'-1}t^{d'}$ be two 
$f$-polynomials of simplicial complexes. Then the product $f(t)g(t)$ and the Hadamard product $f(t)\circ g(t)=1+
f_0g_0\,t + f_1g_1 \,t^ 2+ \cdots + f_{\min\{d,d'\}}g_{\min\{d,d'\}}\,t^ {\min\{d,d'\}+1}$ are $f$-polynomials of simplicial complexes.
\end{proposition}
\begin{proof}
    Let $\Delta$ be a simplicial complex over
    ground set $\Omega$ with
    $f$-vector $(1,f_0,\ldots,f_{d-1})$ and let 
    $\Delta'$ be a simplicial complex over
    ground set $\Omega'$ with 
    $f$-vector $(1,g_0,\ldots,g_{d'-1})$. 
 We can assume that $\Omega$ and $\Omega'$ are disjoint.
 
    \noindent {\sf Ordinary Product:} 
   The join $\Delta * \Delta'$ of
    $\Delta$ and $\Delta'$ is defined as
    the simplicial complex whose $i$-dimensional
    simplices are unions $\sigma \cup \tau$ for 
    a $j$-dimensional simplex $\sigma \in \Delta$
    and $(i-j-1)$-dimensional simplex $\tau$ of $\Delta'$
    for some $-1 \leq j \leq i$.
    By disjointness of $\Omega$ and $\Omega'$ it follows
    that the number of $i$-dimensional faces of 
    $\Delta*\Delta'$ is $\sum_{j=-1}^i f_jg_{i-j-1}$.
    The latter expression is exactly the coefficient of
    $t^i$ in $f(t)g(t)$. In particular, $f(t)g(t)$ is the $f$-polynomial of a simplicial complex.
    
    \noindent {\sf Hadamard Product:} We assume in addition that $\Omega$ and
    $\Omega'$ are linearly ordered by $<$ and $<'$ respectively. 
    Let $\Delta \circ \Delta'$ be
    the simplicial complex over
    ground set $\Omega \times \Omega'$ where
    $\{(\omega_1,\omega_1'),\ldots, 
    (\omega_i,\omega_i')\}$ is a simplex in $\Delta \cdot \Delta'$
    if any only if 
    \begin{itemize}
    \item $\{\omega_1,\ldots,\omega_i\} \in \Delta$
    and $\{\omega_1',\ldots, \omega_i'\}\in \Delta'$
    and
    \item $\omega_1 < \cdots < \omega_i$ and
    $\omega_1' <' \cdots <' \omega_i'$.
    \end{itemize}
    One easily checks that $\Delta \circ \Delta'$
    is a simplicial complex with $f$-vector
    $(f_{-1}g_{-1},f_0g_0,\cdots)$.
\end{proof}

As an example we apply the preceding proposition to
 Narayana polynomials of type B
$N_n^B(t)= \binom{n}{0}^2 + \binom{n}{1}^ 2 \, t + \cdots 
+ \binom{n}{n}^2\,t^n$. It is well known (see Example 5.8 in \cite{FoRe}) that $N_n^B(t)$ is the $h$-polynomial of the simplicial complex dual to the associahedron of type $B$ -- the cyclohedron. As a Hadamard product 
$N_n^ B(t) = \Big(\, (1+t)^ n\,\Big) \circ \Big(\, (1+t)^ n\,\Big)$
of
a real-rooted polynomial with itself, $N_n^B(t)$ is real-rooted for 
any integer $n\geq 1$.
Since $(1+t)^ n$ is the $f$-polynomial of the simplex of
dimension $n-1$, the proposition implies that $N_n^ B(t)$ is the
$f$-polynomial of a simplicial complex.

\section*{Acknowledgments}
We thank Petter Br\"and\'en for informing us about the connection of an 
earlier version of \ref{xk-decreasing} and Lorentzian polynomials. We also thank Aryaman Jal for helpful comments. 

We are grateful to the referee for the detailed report, as well as for numerous valuable suggestions. 
All of these comments have contributed very much to improving the article.

The first author was partially supported by the National Natural Science Foundation of China (Grant Nos. 12271222), Scientific Research Foundation of Jiangsu Normal University (Grant Nos. 21XFRS019) and the China Scholarship Council.

\end{document}